\documentclass[12pt]{amsart}

\usepackage[mathcal]{euscript}
\usepackage{amssymb, amsfonts, xypic, amsmath}
\usepackage{amscd}
\usepackage[all]{xy}
\usepackage{dutchcal}
\usepackage{amsthm}
\usepackage[margin=1.5in]{geometry}
\usepackage{xcolor}
\usepackage{tikz-cd}

\usepackage{amssymb}

\usepackage{enumitem}

\usepackage{graphicx}

\makeatletter
\@namedef{subjclassname@2020}{%
  \textup{2020} Mathematics Subject Classification}
\makeatother

\usepackage[T1]{fontenc}

\newtheorem{theorem}{Theorem}[section]
\newtheorem{question}[theorem]{Question}
\newtheorem{corollary}[theorem]{Corollary}

\newtheorem{proposition}[theorem]{Proposition}

\theoremstyle{definition}

\newtheorem{remark}[theorem]{Remark}

\numberwithin{equation}{section}

\usepackage{graphicx}
\usepackage{epstopdf}

\begin{document}


\baselineskip=17pt

\DeclareGraphicsRule{.tif}{png}{.png}{`convert \#1 `basename \#1
.tif`.png}
\author{Manuel Rivera and Mahmoud Zeinalian}

\newcommand{\Addresses}{{
  \bigskip
  \footnotesize

    \textsc{Manuel Rivera, Department of Mathematics, Purdue University, 150 N. University Street, West Lafayette, IN 47907-2067}\par\nopagebreak \textit{E-mail address} \texttt{manuelr@purdue.edu}

  \medskip
  \medskip

  \textsc{Mahmoud Zeinalian, Department of Mathematics, City University of New York, Lehman College, 250 Bedford Park Blvd W, Bronx, NY 10468
   }\par\nopagebreak
  \textit{E-mail address} \texttt{mahmoud.zeinalian@lehman.cuny.edu}

}}

\begin{abstract}
We show that the natural algebraic structure of the singular chains on a path connected topological space determines the isomorphism class of the fundamental group functorially. Moreover, we describe a notion of weak equivalence for the relevant algebraic structure under which the isomorphism class of the fundamental group is preserved. 
\end{abstract}

\title[Singular chains and the fundamental group]{Singular chains and the fundamental group}



\maketitle

\section{Introduction}

One knows in algebraic topology that the homotopical properties of path-connected topological spaces can be recast into the language of topological monoids by associating to a space its based loop space. The based loop space is a topological monoid consisting of all loops based at a fixed point in the underlying space equipped with a continuous multiplication given by concatenating loops. Furthermore, every loop has an inverse up to homotopy and one can show that the based loop space has the homotopy type of a topological group. This construction relates two classical algebraic invariants both of which date back to the work of Poincar\'e: homology groups and the fundamental group. For example, the 0-th homology group of the based loop space is an algebra which can be identified with the group algebra of the fundamental group. Any group algebra has a compatible coproduct operation making it into a structure known as a bialgebra. The bialgebra structure of a group algebra is in fact a Hopf algebra, namely, it has a linear endomorphism, called the antipode, satisfying a compatibility equation. If a bialgebra admits an antipode then such an endomorphism is unique. The group algebra, when considered as a Hopf algebra, determines the underlying group functorially through the ``group-like elements" functor. 

The main result of this article says that the Hopf algebra determining the fundamental group and higher dimensional aspects can be determined, in complete generality, directly from the natural algebraic structure of the chain complex of singular chains on the underlying space. We also provide a notion of weak equivalence for the algebraic structure under consideration through which the isomorphism class of the fundamental group is preserved functorially. The new idea, beyond the technical details, is to combine 1) the bar-cobar duality theory for algebraic structures, 2)  the homotopical symmetry of chain approximations to the diagonal map on a space (also known as the $E_{\infty}$-coalgebra structure on chains extending the Alexander-Whitney diagonal), and 3) an extension from simply connected spaces to path-connected spaces of a classical result of Frank Adams which relates the based loop space and the cobar functor \cite{Ad56}. 

A more precise description of our results is the following. For any path connected pointed space $(X,x)$ we describe a differential graded (dg) bialgebra structure on the cobar construction of the dg coalgebra of normalized (pointed) singular chains with Alexander-Whitney coproduct. This dg bialgebra structure on the cobar construction of the singular chains on $X$ models the chains on the based loop space of $X$. Hence, this dg bialgebra has the property that the induced bialgebra structure on the $0$-th homology is a Hopf algebra, more specifically, it is a group algebra. The fundamental group can be recovered by applying the group-like elements functor to such a Hopf algebra. The dg bialgebra structure on the cobar construction of the dg coalgebra of singular chains is determined by explicit maps part of an algebraic package describing the homotopical symmetry of the Alexander-Whitney diagonal approximation. 

The notion of weak equivalence, under which the isomorphism class of the fundamental group is preserved, is called $\Omega$-quasi-isomorphism. An $\Omega$-quasi-isomorphism of dg coassociative coalgebras, possibly with extra algebraic structure, is defined to be a morphism which induces a quasi-isomorphism after applying the cobar functor to the underlying dg coassociative coalgebra structure. Every $\Omega$-quasi-isomorphism of dg coalgebras is a quasi-isomorphism but not vice-versa. 

One of the fundamental problems in algebraic topology is to associate natural algebraic invariants to a certain class of topological spaces that determine those spaces up to a specified notion of equivalence. The rational homotopy theory of Quillen and Sullivan provide an answer for the class of simply connected spaces up to rational homotopy equivalence. The algebraic invariant associated to a space, in this case, is extracted from the graded commutative algebra structure of the rational cochains. 

Mandell, Goerss, Karoubi, Kriz, Blomquist and Harper, among others, have studied $p$-adic and integral versions of the problem by generalizing the constructions of rational homotopy theory. Most approaches are based on the existence of an $E_{\infty}$-algebra structure on the cochains of a space extending the cup product. The construction of such algebraic structure comes from dualizing an $E_{\infty}$-coalgebra structure at the level of chains extending the Alexander-Whitney coproduct. The statements found in the literature relating weak homotopy types and $E_{\infty}$-structures require a simple connectivity or nilpotency hypothesis on the class of spaces being considered. 

Our motivation is to extend these classical results regarding complete algebraic invariants for simply connected or nilpotent homotopy types to the class of connected homotopy types. More precisely, we would like to understand to what extent the $E_{\infty}$-coalgebra structure of singular chains, with the correct piece of extra structure, under $\Omega$-quasi-isomorphisms determines the integral homotopy type of a path connected space. This should lead to incorporating the fundamental group into the algebraic package of \cite{Ma06}. We hope this note serves as a first step towards this goal.

In section 2, we introduce notation, recall basic constructions, and make precise statements. In section 3, we discuss the proofs of our statements using language and basic results from higher category theory. 

\section{Main statements}

We work at the level of categories with weak equivalences. A category with weak equivalences is a pair $(\mathcal{C}, \mathcal{W})$ where $\mathcal{C}$ is an ordinary category and $\mathcal{W} $ is a distinguished class of morphisms in $\mathcal{C}$ called weak equivalences which is closed under composition, contains all isomorphisms, and satisfies the following two out of three property: for any two composable morphisms $f$ and $g$ in $\mathcal{C}$, if two elements of $\{f, g, g \circ f\}$ are in $\mathcal{W}$, then so is the third. A functor between categories with weak equivalences is a functor of ordinary categories which preserves the weak equivalences.

Fix a commutative ring $\mathbf{k}$ with unit $1_{\mathbf{k}}$. Let $\mathcal{Set}_{\Delta}=\text{Fun}(\Delta^{op}, \mathcal{Set})$ be the category of simplicial sets and denote by $w.h.e.$ the class of weak homotopy equivalences (also known as Kan weak equivalences) i.e. morphisms $f: S \to S'$ of simplicial sets for which the map $|f|: |S| \to |S'|$ between geometric realizations induces an isomorphism on homotopy groups. Associated to any $S \in \mathcal{Set}_{\Delta}$ there is a natural differential graded (dg) coalgebra $C(S)=(C_*(S), \partial, \Delta)$ over $\mathbf{k}$ of normalized simplicial chains with coproduct $\Delta$ known as the Alexander-Whitney diagonal approximation. Recall that for any $n \geq 0$,  $C_n(S)$ is $\mathbf{k}$-module freely generated by the elements of $S_n$  modulo degenerate simplices. For any generator $\sigma \in C_n(S)$, the differential $\partial: C_n(S) \to C_{n-1}(S)$ is induced by the formula $\partial(\sigma)= \sum_{i=0}^n (-1)^iS(d_i)(\sigma)$ and the coproduct $\Delta: C(S) \to C(S) \otimes C(S)$ by
\[ \Delta(\sigma)= \bigoplus_{p+q=n} S(f_p)(\sigma) \otimes S(l_q)(\sigma),\]
where $d_i: [n-1] \to [n]$ is the inclusion that ``skips" $i$ (the $i$-th coface map), $f_p: [p] \to [p+q]$ is $f_p(i)=i,$ and $l_q: [q] \to [p+q]$ is $l_q(i)=p+i$. 

Let $\mathcal{Set}_{\Delta}^0 \subset \mathcal{Set}_{\Delta}$ be the full subcategory whose objects are simplicial sets $S$ such that $S_0$ is a singleton. The objects of $\mathcal{Set}_{\Delta}^0$ are called $0$-reduced simplicial sets. The normalized chains construction defines a functor 
$$C: (\mathcal{Set}^0_{\Delta}, w.h.e. ) \to (\mathcal{C}_{\mathbf{k}}, q.i.)$$
where $\mathcal{C}_{\mathbf{k}}$ is the category of non-negatively graded connected dg coalgebras and $q.i.$ the class of quasi-isomorphisms. Recall a dg coalgebra $C$ is said to be connected if $C_0\cong \mathbf{k}$ and a quasi-isomorphism of dg coalgebras is a morphism which induces an isomorphism on homology. Whenever we say dg (co)algebra in this article we mean dg (co)associative (co)algebra. Let $\mathcal{Ch}_{\mathbf{k}}$ be the category of non-negatively graded chain complexes of $\mathbf{k}$-modules. It will be useful to denote by
$$\bar{C}: \mathcal{Set}_{\Delta} \to \mathcal{Ch}_{\mathbf{k}}$$
the normalized simplicial chains functor $\bar{C}(S)= (C_*(S), \partial)$ without the coalgebra structure.

Recall the definition of the cobar functor
$$\Omega: \mathcal{C}_{\mathbf{k}} \to \mathcal{A}_{\mathbf{k}},$$
from the category of connected dg coalgebras to the category of augmented dg algebras. For any connected dg coalgebra $(C, \partial_C:C \to C, \Delta_C:C\to C \otimes C)$ define a dg algebra
$$\Omega C := ( T(s^{-1}  \tilde{C} ), D, \mu)$$
where $\tilde{C}_0:=0$ and $\tilde{C}_i=:C_i$ for $i>0$, $s^{-1}$ is the shift by $-1$ functor, $T(s^{-1} \tilde{C})= \bigoplus_{i=0}^{\infty} (s^{-1}\tilde{C})^{\otimes i}$ the tensor algebra with product $\mu$ given by concatenation of monomials, and the differential $D$ is defined by extending the linear map $$- s^{-1} \circ \partial_C \circ s^{+1} + (s^{-1} \otimes s^{-1}) \circ \Delta_C \circ s^{+1}: s^{-1}\tilde{C} \to s^{-1}\tilde{C} \oplus (s^{-1}\tilde{C} \otimes s^{-1}\tilde{C})$$ as a derivation to obtain a map $D: T(s^{-1} \tilde{C}) \to T(s^{-1} \tilde{C})$. The coassociativity of $\Delta$, the compatibility of $\partial$ and $\Delta$, and the fact that $\partial^2 =0$ together imply that $D^2=0$. The projection $T(s^{-1}\tilde{C}) \to (s^{-1}\tilde{C})^{\otimes 0}=\mathbf{k}$ defines an augmentation. The construction of $\Omega$ is clearly natural with respect to maps of dg coalgebras and thus defines a functor.

Let $\Omega\text{-}q.i.$ be the class of morphisms $f: C \to C'$ in $\mathcal{C}_{\mathbf{k}}$ for which the induced dg algebra map at the level of cobar constructions $\Omega f: \Omega C \to \Omega C'$ is a quasi-isomorphism. We call the morphisms in $\Omega\text{-}q.i.$ $\Omega$-quasi-isomorphisms. Any $\Omega$-quasi-isomorphism of dg coalgebras is a quasi-isomorphism but not viceversa. Hence, there is an inclusion functor $\iota: (\mathcal{C}_{\mathbf{k}}, \Omega\text{-}q.i.) \to (\mathcal{C}_{\mathbf{k}}, q.i.)$. There is a model category structure on the category of dg coalgebras having $\Omega\text{-}q.i.$ as weak equivalences, but this fact will not be used in this article. 

Let $K: (\mathcal{Set}^0_{\Delta}, w.h.e.) \to (\mathcal{Set}^0_{\Delta}, w.h.e.)$ be the functor $K(S)= \text{Sing}(|S|,x)$ where $S_0=\{x\}$, $|S|$ the geometric realization of $S$, and $\text{Sing}(|S|,x)$ denotes the Kan complex consisting of singular simplices $|\Delta^n| \to |S|$ sending all the vertices of $|\Delta^n|$ to $x$. $K$ is a fibrant replacement functor in the Quillen model category of simplicial sets. Our first observation is the following. 

\begin{theorem} \label{thm1} There exists a functor between categories with weak equivalences
$$C^K: (\mathcal{Set}^0_{\Delta}, w.h.e. ) \to (\mathcal{C}_{\mathbf{k}}, \Omega\text{-}q.i.)$$ such that 
\\
(a) $\iota \circ C^K = C \circ K$, and
\\
(b) for any $S \in Set_{\Delta}^0$ there is a quasi-isomorphism of dg algebras $$\Omega (C^K(S)) \simeq \bar{C}(\emph{Sing}( \Omega_x|S|)),$$ where $\Omega_x|S|$ denotes the topological monoid of (Moore) loops in $|S|$ based at $x$ with product induced by concatenation of loops and $\bar{C}(\emph{Sing}( \Omega_x|S|))$ the dg algebra of singular chains on $\Omega_x|S|$.
\end{theorem} 
Note that \textit{(a)} is equivalent to saying that $C \circ K$ sends weak homotopy equivalences to the subclass of $\Omega$-quasi-isomorphisms inside quasi-isomorphisms. 

\begin{remark} In general, $\Omega (C(S))$ is not quasi-isomorphic to $\Omega (C^K(S))$ even though $C(S)$ and $C^K(S)$ are quasi-isomorphic as dg coalgebras (the cobar functor $\Omega$ does not preserve quasi-isomorphisms). By a classical theorem of Adams \cite{Ad56}, there is a quasi-isomorphism $\Omega (C(S)) \simeq \bar{C}(\text{Sing}( \Omega_x|S|))$ of dg algebras for any $1$-reduced simplicial set $S$, namely, for any $S$ with no non-degenerate $1$-simplices. Moreover, the notions of quasi-isomorphism and $\Omega$-quasi-isomorphism are equivalent when restricted to the category of simply-connected dg coalgebras, i.e. connected dg coalgebras $C$ such that $C_1=0$. Hence, Theorem \ref{thm1} may be regarded as an extension of Adams' theorem.
\end{remark}

We construct a lift of $C^K$ to a new category $\mathcal{B}_{\mathbf{k}}$ with weak equivalences $\Omega\text{-}q.i.$ by adding more algebraic structure on the chains that determine the fundamental group. An object  $\mathbf{C}$ in $\mathcal{B}_{\mathbf{k}}$  consists of the data  $\mathbf{C}= (C, \partial, \Delta, \nabla, \eta, \epsilon)$ where $(C, \partial: C \to C, \Delta: C \to C \otimes C)$ is a connected dg coalgebra over $\mathbf{k}$, $\nabla: \Omega C \to \Omega C \otimes \Omega C$ is a coassociative coproduct on the cobar construction of $C$ making $\Omega C$ a dg bialgebra (i.e. $(\Omega C, D,\nabla)$ is a dg coalgebra and $\nabla$ is an algebra map) with unit $\eta: \mathbf{k} \to \Omega C$, counit $\epsilon: \Omega C \to \mathbf{k}$, and with the property that $H_0(\Omega C)$ is a Hopf algebra.\footnote{A bialgebra $(H, \mu: H \otimes H \to H, \nabla: H \to H\otimes H, \eta: \mathbf{k} \to H, \epsilon: H \to \mathbf{k})$ is called a Hopf algebra if there exists a linear map $s: H \to H$ satisfying  $\mu \circ (s \otimes id) \circ \nabla = \eta \circ \epsilon = \mu \circ (id \otimes s) \circ \nabla$. When such a map $s$ exists it is unique and is called the antipode of $H$.}  A morphism $f: \mathbf{C} \to \mathbf{C}'$ in $\mathcal{B}_{\mathbf{k}}$ is a map of underlying dg coalgebras $f: C \to C'$ for which $\Omega f: \Omega C \to \Omega C'$ is a dg bialgebra map. Since maps of bialgebras automatically preserve antipodes it follows that $H_0(\Omega f) : H_0(\Omega C) \to H_0(\Omega C')$ a Hopf algebra map. The class of weak equivalences $\Omega\text{-}q.i.$ consists of those morphisms $f: \mathbf{C} \to \mathbf{C'}$ for which $\Omega f: \Omega C \to \Omega C'$ is a quasi-isomorphism. 

\begin{remark} $A_{\infty}$-algebras with a dg bialgebra structure on its bar construction may be found in the literature under the name of $B_{\infty}$-algebras and $B_{\infty}$-coalgebras are defined dually. Thus $\mathcal{B}_{\mathbf{k}}$ is a subcategory of the category of $B_{\infty}$-coalgebras.
\end{remark}

Forgetting the extra algebraic structure we obtain a functor $F: (\mathcal{B}_{\mathbf{k}}, \Omega\text{-}q.i.) \to (\mathcal{C}_{\mathbf{k}}, \Omega\text{-}q.i.)$. We also consider the functors $H_{\Omega} : (\mathcal{B}_{\mathbf{k}}, \Omega\text{-}q.i.) \to (\mathcal{H}_{\mathbf{k}}, iso.)$ and $Grp: (\mathcal{H}_{\mathbf{k}}, iso.) \to (\mathcal{G}, iso.)$, where $(\mathcal{H}_{\mathbf{k}}, iso.)$ denotes the category of Hopf algebras with isomorphisms and $(\mathcal{G}, iso.)$  the category of groups with isomorphisms, defined on objects by $H_{\Omega}(\mathbf{C}):=H_0(\Omega C)$ and 
$$Grp (H, \mu, \nabla, \eta, \epsilon) :=\{g \in H: \epsilon(g)= 1_{\mathbf{k}} \text{ and }\nabla(g)= g \otimes g \},$$  respectively. These functors are defined on morphisms in the obvious way. $Grp$ is called the group-like elements functor. 

Let $\pi_1: (\mathcal{Set}_{\Delta}^0, w.h.e.) \to (\mathcal{G}, iso.)$ be the functor given by $\pi_1(S):=\pi_1(|S|,x)$, for any $S \in \mathcal{Set}_{\Delta}^0$ with $S_0=\{x\}$, where $\pi_1(|S|,x)$ denotes the fundamental group of the pointed space $(|S|,x)$.

\begin{theorem} \label{thm2}
There exists functors between categories with weak equivalences $$B: (\mathcal{Set}_{\Delta}^0, w.h.e.)  \to (\mathcal{B}_{\mathbf{k}}, \Omega\text{-}q.i.)$$ 
and
$$F: (\mathcal{B}_{\mathbf{k}}, \Omega\text{-}q.i.) \to  (\mathcal{C}_{\mathbf{k}}, \Omega\text{-}q.i.)$$
such that  
\\
(a) $F \circ B= C^K$, and 
\\
(b) there is a natural isomorphism of functors $\pi_1 \cong  Grp \circ H_{\Omega} \circ B$. 
\end{theorem}

\begin{remark} In the case when $C=C(S)$ for $S$ a $1$-connected simplicial set, so $\Omega C$ is a connected dg algebra, a coproduct $\nabla: \Omega C \to \Omega C \otimes \Omega C$ was previously constructed by Baues in \cite{Ba98}. In this article, we will describe how to extend this dg bialgebra structure to any $0$-connected simplicial set $S$.
\end{remark}

The dg bialgebra structure on the cobar construction of the dg coalgebra $C(S)$ may be understood as part of the $E_{\infty}$-coalgebra structure of $C(S)$ describing the homotopical symmetry of the Alexander-Whintey diagonal approximation. Let $\chi$ be the surjection dg operad as described in \cite{BeFr04}. The surjection operad $\chi$ is an $E_{\infty}$-operad and there is a natural $\chi$-coalgebra structure on the normalized chains $C(S)$ of any simplicial set $S$ extending the Alexander-Whitney coproduct. This rich algebraic structure on the chains of a simplicial set has been described in \cite{BeFr04}, \cite{McSm02}, and others. There is a natural map of operads $\rho: As \to \chi$, where $As$ is the associative operad, and for every $\chi$-coalgebra $\mathbf{C}$ we denote by $C$ the dg coalgebra ($As$-coalgebra) obtained by pulling back the action along $\rho$. We call $C$ the underlying dg coalgebra of the $\chi$-coalgebra $\mathbf{C}$. 

By Remark 1 in \cite{Ka03}, for any $\chi$-coalgebra there is an underlying homotopy $G$-coalgebra structure. In section 2.2 of \cite{Ka03} it is explained how any homotopy $G$-coalgebra gives rise to a dg bialgebra structure on the cobar construction of the underlying dg coassociative coalgebra. Hence, for any $\chi$-coalgebra $\mathbf{C}$ there is an induced dg bialgebra structure on $\Omega C$. We refer to \cite{Ka03} for the definition of a homotopy $G$-coalgebra structure and further details.

Let $\mathcal{E}_{\mathbf{k}}$ be the category of $\chi$-coalgebras consisting of $\chi$-coalgebras $\mathbf{C}$ whose underlying dg coalgebra $C$ has the property that $H_0(\Omega C)$ admits an antipode making it a Hopf algebra. Any map $f: \mathbf{C} \to \mathbf{C}'$ of $\chi$-coalgebras induces a map of dg bialgebras $\Omega f: \Omega C \to \Omega C$ which in turn induces a map of Hopf algebras $H_0(\Omega f): H_0(\Omega C) \to H_0(\Omega C)$. Denote by $\Omega\text{-}q.i.$ the class of those maps $f: \mathbf{C} \to \mathbf{C}'$ in $\mathcal{E}_{\mathbf{k}}$ for which  $\Omega f: \Omega C \to \Omega C$ is a quasi-isomorphism.

The following statement says that we may further lift $B: (\mathcal{Set}_{\Delta}^0, w.h.e.)  \to (\mathcal{B}_{\mathbf{k}}, \Omega\text{-}q.i.)$ to $(\mathcal{E}_{\mathbf{k}}, \Omega\text{-}q.i.)$ resulting in a functor which encodes both the $\chi$-coalgebra structure and the fundamental group.

\begin{theorem} \label{thm3}
There exists functors between categories with weak equivalences $$E : (\mathcal{Set}_{\Delta}^0, w.h.e.) \to (\mathcal{E}_{\mathbf{k}}, \Omega\text{-}q.i.)$$
and
$$G :(\mathcal{B}_{\mathbf{k}}, \Omega\text{-}q.i.)\to (\mathcal{E}_{\mathbf{k}}, \Omega\text{-}q.i.)$$
such that 
\\
(a) for any $S \in \mathcal{Set}_{\Delta}^0$, $E(S)$ is the $\chi$-coalgebra extending the dg coalgebra $C(K(S))$ as described in \cite{BeFr04}, and
\\
(b) there is a natural isomorphism of functors $G \circ E \cong B$. 
\end{theorem}
\begin{remark}

We could have combined Theorems \ref{thm2} and \ref{thm3} into a single statement. However, we have decided to separate it into two statements in order to elucidate the fact that we do not need the full $E_{\infty}$-coalgebra structure of the singular chains to recover the fundamental group but only part of it. 

\end{remark}
Note that Theorem 2, together with $G \circ E \cong B$, implies that there is a natural isomorphism of functors $\pi_1 \cong Grp \circ H_{\Omega} \circ G \circ E$.
Hence, for any $S \in \mathcal{Set}_{\Delta}^0$, $E(S)$ knows the fundamental group of $|S|$ and the $E_{\infty}$-coalgebra structure on the chains of $S$. Moreover, this data may be recovered from any $\mathbf{C} \in \mathcal{E}_{\mathbf{k}}$  which is $\Omega$-quasi-isomorphic to $E(S)$. We finish this section by posing the question which is the main motivation for our work.

\begin{question} Let $\mathbf{k}=\mathbb{Z}$ and $S, S' \in Set^0_{\Delta}$. Are the $E_{\infty}$-coalgebras $E(S)$ and $E(S')$ $\Omega$-quasi-isomorphic if and only if $S$ and $S'$ are weakly homotopy equivalent?
\end{question}
We believe the answer to this question should be positive under mild conditions. We propose the following outline:
\\

\textit{Step 1.} For any dg coassociative coalgebra $C$ we may form a twisted tensor product complex $C \otimes_{\iota} H_0(\Omega C)$  via the universal twisting cochain $\iota: C \hookrightarrow \Omega C$ by considering $H_0(\Omega C)$ as a left $\Omega C$-module. When $C$ is the dg coalgebra of singular chains on $|S|$, the twisted tensor product $C \otimes_{\iota} H_0(\Omega C)$ is a chain complex model for the universal cover of $|S|$. If $C$ is the underlying dg coassociative coalgebra of a $\chi$-coalgebra $\mathbf{C}$ construct a natural $E_{\infty}$-coalgebra structure on the chain complex $C \otimes_{\iota} H_0(\Omega C)$ such that when $C=C(S)$ such a structure is equivalent to the $E_{\infty}$-coalgebra structure of the chains on the universal cover of $|S|$. This construction will describe the $E_{\infty}$-coalgebra structure of the chains on the universal cover of a space $X$ in terms of the $E_{\infty}$-coalgebra structure of the chains on $X$. 
\\

\textit{Step 2.} Show that if $\mathbf{C}$ and $\mathbf{C'}$ are $\Omega$-quasi-isomorphic $\chi$-coalgebras then their twisted tensor products $C \otimes_{\iota} H_0(\Omega C)$ and $C' \otimes_{\iota} H_0(\Omega C')$ are quasi-isomorphic as $E_{\infty}$-coalgebras in an equivariant sense with respect to the right $H_0(\Omega C)$-module structure.  
\\

\textit{Step 3.} Apply an equivariant version of Mandell's theorem for simply connected spaces to conclude that if $E(S)$ and $E(S')$ are $\Omega$-quasi-isomorphic then the universal covers of $S$ and $S'$ are weakly homotopy equivalent in a way compatible with the action of the fundamental group $\pi_1(S) \cong \pi_1(S')$. Then deduce that the weak homotopy equivalence between universal covers descends to a weak homotopy equivalence between the base spaces.
\\

Theorems 1,2, and 3 describe a sequence of lifts of the functor of simplicial chains with Alexander-Whitney coproduct as described by the following commutative diagram.

\vspace{35pt}
\phantom{aaaaaaa}
\begin{tikzcd}[column sep=large]
& \phantom{aaaaa}(\mathcal{E}_{\mathbf{k}}, \Omega\text{-}q.i.)\arrow[d, shift left=5, ""]\\
& \phantom{aaaaa}(\mathcal{B}_{\mathbf{k}}, \Omega\text{-}q.i.)\arrow[d, shift left=5, ""]\\
& \phantom{aaaaa}(\mathcal{C}_{\mathbf{k}}, \Omega\text{-}q.i.)\ar[d, shift left=5, ""]\\
\phantom{sssssssssss}(\mathcal{Set}^0_{\Delta}, w.h.e.)\phantom{aaaaaaa} \ar[r, "C \circ K", ]\ar[ru, "C^K", very near end]\arrow[ruu,  "B", near end]\arrow[ruuu,  "E", near end]& \phantom{aaaaa}(\mathcal{C}_{\mathbf{k}},q.i.)
\end{tikzcd}\\

\section{Proofs}
In this section we prove theorems 1,2, and 3. We use the notation introduced in the previous section as well as language and basic results from the theory of higher categories. 

\subsection{Theorem 1: the cobar construction as a dg algebra model for the based loop space of a path connected space} Theorem 1 follows from constructions and results of \cite{RiZe16}, where we explain the relationship between Lurie's rigidification functor $\mathfrak{C}: \mathcal{Set}_{\Delta} \to \mathcal{Cat}_{\Delta}$ (we recall its definition below), the cobar functor $\Omega: \mathcal{C}_{\mathbf{k}} \to \mathcal{A}_{\mathbf{k}}$, and the based loop space. In this subsection we give a condensed and relatively self contained presentation, following \cite{RiZe16}, explaining how Theorem 1 may be deduced from properties of the functor $\mathfrak{C}$.

Denote by $\mathcal{Cat}_{\Delta}$ the category of simplicial cateogries, i.e. categories enriched over the monoidal category $(\mathcal{Set}_{\Delta}, \times)$ of simplicial sets with cartesian product. Denote by $ho: \mathcal{Cat}_{\Delta} \to \mathcal{Cat}$ the functor induced by applying $\pi_0: \mathcal{Set}_{\Delta} \to \mathcal{Set}$ at the level of morphisms. For any $\mathcal{C} \in \mathcal{Cat}_{\Delta}$, $ho(\mathcal{C}) \in \mathcal{Cat}$ is called the homotopy category of $\mathcal{C}$. A morphism $F: \mathcal{C} \to \mathcal{D}$ of simplicial categories is called a weak equivalence if $ho(F): ho(\mathcal{C}) \to ho(\mathcal{D})$ is an essentially surjective functor of ordinary categories and for any two objects $x,y \in \mathcal{C}$, $F: \mathcal{C}(x,y) \to \mathcal{D}(F(x), F(y))$ is a weak homotopy equivalence of simplicial sets.

We recall the definition of 
\begin{equation}
\mathfrak{C}: \mathcal{Set}_{\Delta} \to \mathcal{Cat}_{\Delta}
\end{equation} following \cite{DuSp11}. Given any $S \in \mathcal{Set}_{\Delta}$ define a simplicial category $\mathfrak{C}(S)$ which has as objects the elements of $S_0$ and for any $x , y \in S_0$ the simplicial set of morphisms is defined by glueing a collection of simplicial cubes according to the following colimit in $Set_{\Delta}$

\begin{equation} \label{necs}
\mathfrak{C}(S)(x,y) := \underset{T \to S \in (Nec \downarrow S)_{x,y} }{\text{colim}}  (\Delta^{1})^{\times V_T} .
\end{equation}
We explain the above formula. The category of necklaces, $Nec$, has as objects simplicial sets of the form $T= T^1 \vee ... \vee T^k$ where $T^i = \Delta^{t_i}$ for some integer $t_i \geq 1$ and the identifications in the wedge product are given by glueing the last vertex of $T^i$ with the first vertex of $T^{i+1}$ for $i=1,...,k-1$. Since the vertices of each $T^i$ are ordered, the set $T_0$ of vertices in $T$ has a natural ordering. The morphisms in $Nec$ are defined to be maps $f: T \to T'$ of simplicial sets which preserve first and last vertices, i.e. satisfying $f(\alpha_T)=\alpha_{T'}$ and $f(\omega_T)=\omega_{T'}$, where $\alpha_T$ and $\omega_T$ denote the first and last vertices of $T$, respectively. The category $(Nec \downarrow S)_{x,y}$ is the full subcategory of the over category $Nec \downarrow S$ whose objects are those maps $g: T \to S$ satisfying $g(\alpha_T)=x$ and $g(\omega_T)=y$. To any $T= T^1 \vee ... \vee T^k \in Nec$ we may associate functorially a simplicial cube $(\Delta^1)^{\times V_T}$, where $V_T= t_1+...+t_k-k$, as follows. Consider the functor $Nec \to Set_{\Delta}$ which sends a necklace $T \in Nec$ to the simplicial set $N(P_T) \in Set_{\Delta}$, where $N(P_T)$ is the nerve of the category $P_{T}$ whose objects are subsets $U \subset T_0$ such that $\alpha_{T_0} \in U$ and $\omega_{T^i} \in U$ for all $i=1,...,k$, and morphisms are inclusions of sets. The assignment $T \mapsto N(P_T)$ is natural with respect to morphisms in $Nec$ and $N(P_T)$ is naturally isomorphic to $(\Delta^1)^{\times V_T}$. 

Composition in $\mathfrak{C}(S)$ is given as follows. For any $x,y \in S_0$, each simplex $\sigma \in \mathfrak{C}(S)(x,y)_n$ is given by an equivalence class $\sigma= [f: T \to S, \xi]$ where $(f: T\to S) \in (Nec \downarrow S)_{x,y}$ and $\xi\in ((\Delta^1)^{\times {V_T}})_n$. Then the composition rule
$$(\mathfrak{C}(S)(y,z) \times \mathfrak{C}(S)(x,y))_n   = \mathfrak{C}(S)(y,z)_n \times \mathfrak{C}(S)(x,y)_n \to \mathfrak{C}(S)(x,z)_n$$
is given by 
$$[f': T' \to S, \xi'] \times [f: T \to S, \xi] \mapsto [ f \vee f': T \vee T' \to S, \xi \times \xi'].$$

\begin{remark}  The functor  $\mathfrak{C}: \mathcal{Set}_{\Delta} \to \mathcal{Cat}_{\Delta}$ is defined in Chapter 1.1.5 of \cite{Lu09} by first declaring $\mathfrak{C}(\Delta^n) \in \mathcal{Cat}_{\Delta}$ to have the set $\{0,...,n\}$ as objects and, if $i<j$, $\mathfrak{C}(S)(i,j)$ is defined as the nerve of the category $P_{i,j}$ whose objects are subsets $U \subseteq \{i, i+1,..., j\}$ and morphisms are inclusions of sets, $\mathfrak{C}(S)(i,i)=\Delta^0$, and empty otherwise. This data defines a simplicial category $\mathfrak{C}(\Delta^n)$ with composition induced by taking union of subsets. Then extend to any arbitrary simplicial set $S$ in the usual way by $$\mathfrak{C}(S):=  \underset{\Delta^n \to S}{\text{colim}} \mathfrak{C}(\Delta^n).$$ It follows from Theorem 1.4 and Corollary 4.4 of \cite{DuSp11} that the definition given above is equivalent to Lurie's definition for $\mathfrak{C}$. In other words, the definition of $\mathfrak{C}(S)$ as a colimit in $\mathcal{Cat}_{\Delta}$ is unraveled in \cite{DuSp11} by describing explicitly each morphism space as a colimit in $\mathcal{Set}_{\Delta}$. 
\end{remark}

$\mathfrak{C}$ has a right adjoint $$N_{\Delta}: \mathcal{Cat}_{\Delta} \to \mathcal{Set}_{\Delta}$$ called the homotopy coherent nerve functor. Moreover, the adjunction $(\mathfrak{C}, N_{\Delta})$ forms a Quillen equivalence between the Joyal model structure on $\mathcal{Set}_{\Delta}$ and the Bergner model structure on $\mathcal{Cat}_{\Delta}$. In general, $\mathfrak{C}: \mathcal{Set}_{\Delta} \to \mathcal{Cat}_{\Delta}$ does not send weak homotopy equivalences of simplicial sets to weak equivalences of simplicial categories. However, the Quillen model structure is a left Bousfield localization of the Joyal model structure and we have the following.

\begin{proposition} \label{Kan} If $f: S \to S'$ is a weak homotopy equivalence between Kan complexes $S$ and $S'$ then $\mathfrak{C}(f): \mathfrak{C}(S) \to \mathfrak{C}(S')$ is a weak equivalence of simplicial categories.  
\end{proposition}
\begin{proof}
This is Proposition 17.2.8 in \cite{Rie14}.
\end{proof}
We shall now describe how the rigidification functor $\mathfrak{C}: \mathcal{Set}_{\Delta} \to \mathcal{Cat}_{\Delta}$ may be understood as a simplicial version of a functor 
\begin{equation}
\mathcal{P}: \mathcal{Top} \to \mathcal{Cat}_{\mathcal{Top}}
\end{equation} from the category of topological spaces to the category of categories enriched over the monoidal category $(\mathcal{Top}, \times)$. For any $X \in \mathcal{Top}$ define $\mathcal{P}(X)$ to be the topological category whose objects are the points of $X$ and, for any two $x,y \in X$, define 
\begin{equation} \label{paths}
\mathcal{P}(X)(x,y)= \{ (\gamma, r) : \gamma: [0,r] \to X, \gamma(0)=x, \gamma(r)=y \},
\end{equation}
the space of (Moore) paths in $X$ from $x$ to $y$ with the compact open topology. Composition in $\mathcal{P}(X)$ is obtained by concatenating paths. The precise relationship between $\mathfrak{C}$ and $\mathcal{P}$ can be stated in terms of the functor
\begin{equation}
 \textbf{Sing}: \mathcal{Cat}_{\mathcal{Top}} \to \mathcal{Cat}_{\Delta}
 \end{equation}
 defined to be identity on objects and by applying the monoidal functor $\text{Sing}: \mathcal{Top} \to \mathcal{Set}_{\Delta}$  that sends a topological space to its singular Kan complex at the level of morphisms.

\begin{proposition} For any path connected topological space $X$ the simplicial categories $\mathfrak{C} ( \emph{Sing}(X) )$ and $\emph{\textbf{Sing}}( \mathcal{P}(X) )$ are weakly equivalent.
\end{proposition}

\begin{proof}
Fix a point $b$ in $X$. Since $X$ is path connected the topological category $\mathcal{P}X$ is weakly equivalent to the topological category $\mathbf{\Omega} X$ which has a single object $b$ and as morphism space $\mathbf{\Omega} X (b,b) = \mathcal{P}(X)(b,b)$. Let $N_{Top}: \mathcal{Cat}_{\mathcal{Top} } \to \mathcal{Set}_{\Delta}$ be the composition $ N_{\Delta} \circ \textbf{Sing}$. The functor $N_{Top}$, called the topological nerve functor, sends weak equivalences of topological categories to weak homotopy equivalence of simplicial sets. Thus, the simplicial sets $N_{Top}(\mathcal{P}X)$ and  $N_{Top}(\mathbf{\Omega} X)$ are Kan weakly equivalent. Moreover, the geometric realization $|N_{Top}(\mathbf{\Omega}X)|$ is a model for $B(\mathbf{\Omega}X)$, the classifying space of the topological monoid of based loops, so it follows that the spaces $|N_{Top}(\mathbf{\Omega}X)|$ and $X$ are naturally weakly homotopy equivalent. It follows that the simplicial sets $N_{Top}(\mathcal{P}X)$ and $\text{Sing}(X)$ are weakly equivalent as well. On the other hand, $N_{Top}(\mathcal{P}X)$ is a Kan complex (since its homotopy category is a groupoid) so Proposition \ref{Kan} implies there is a weak equivalence $\mathfrak{C}(N_{Top}(\mathcal{P}X)) \simeq \mathfrak{C}(\text{Sing}(X))$. Since $\mathfrak{C} \circ N_{\Delta} (\mathcal{C})  \simeq \mathcal{C}$ for any $\mathcal{C} \in \mathcal{Cat}_{\Delta}$ whose mapping spaces are Kan complexes, we have that $\mathfrak{C} (N_{Top}(\mathcal{P}X)) = \mathfrak{C} (N_{\Delta}( \textbf{Sing}( \mathcal{P}X))) \simeq \textbf{Sing}(\mathcal{P}X)$. Hence, the simplicial categories $\mathfrak{C}(\text{Sing}(X))$ and $\textbf{Sing}(\mathcal{P}X)$ are weakly equivalent. 
\end{proof}

For any path connected space $X$ the inclusion $ \text{Sing}(X,b) \hookrightarrow \text{Sing}(X) $ is a weak homotopy equivalence of Kan complexes, so Proposition 6 implies the following

\begin{corollary} For any path connected topological space $X$ with base point $b \in X$, the simplicial categories $\mathfrak{C}(\emph{Sing}(X,b))$ and $\emph{\textbf{Sing}}(\mathbf{\Omega}X)$ are naturally weakly equivalent.
\end{corollary}

For any $S \in \mathcal{Set}^0_{\Delta}$, $\mathfrak{C}(S)$ is a simplicial category with one object $x$ so the composition rule induces a dg algebra structure on the chain complex $\bar{C}(\mathfrak{C}(S)(x,x))$. Denote this dg algebra by $\Gamma(S)$. 

We proceed by explaining how, for a path connected pointed space $(X,b)$, the dg algebra $\Gamma(\text{Sing}(X,b))$ relates to the cobar construction of the dg coalgebra $C(X):=C(\text{Sing}(X,b))$. The key is to introduce a cubical version of $\mathfrak{C}$ such that upon taking cubical chains on the morphisms we obtain a dg algebra isomorphic to $\Omega C(X)$. We cannot rely on the theory of cubical sets to substitute the simplicial cube $(\Delta^1)^{\times V_T}$ for a standard cube $\square^{V_T}$ in formula 3.1 since not all maps of necklaces are realized by maps of cubical sets. For example, consider the codegeneracy map $s: \Delta^3 \to \Delta^2$ induced by the map of ordered sets  $s: \{0,1,2,3\} \to \{0,1,2\}$ defined by $s(0)=0$, $s(1)=s(2)=1$, and $s(3)=2$. Applying $\mathfrak{C}$ we obtain that $\mathfrak{C}(s): \mathfrak{C}(\Delta^3)(0,3) \to \mathfrak{C}(\Delta^2)(0,2)$ corresponds to a map of simplicial sets $\Delta^1 \times \Delta^1 \to \Delta^1$ which is not realized by a map of cubical sets $\square^2 \to \square^1$. However, $\mathfrak{C}(s)$ may be realized as a map between cubical sets with connections. 

Denote by $\mathbf{1}^n$ the $n$-fold cartesian product of $n$ copies of the category $\mathbf{1}= \{ 0 \to 1 \}$ with two objects and exactly one non-identity morphism. A cubical set with connections is a presheaf over the category $\square_{c}$ whose objects are the categories $\mathbf{1}^n$ and morphisms are generated by the usual cubical co-faces $\delta^{\epsilon}_{j,n}: \mathbf{1}^n \to \mathbf{1}^{n+1}$ ($j=0,1,...,n+1$ and $\epsilon \in \{0,1\}$), cubical co-degeneracies $\varepsilon_{j,n}: \mathbf{1}^n \to \mathbf{1}^{n-1}$ ($j=1,...,n$), together with cubcial co-connections $\gamma_{j,n}: \mathbf{1}^n \to \mathbf{1}^{n-1}$  ($j=1,...,n-1$, $n\geq 2$) defined by
\begin{equation*}
\gamma_{j,n}(s_1,...,s_n)=(s_1,...,s_{j-1},\text{max}(s_j,s_{j+1}),s_{j+2},...,s_n).
\end{equation*}
For a cubical set with connections $Q: \square_c^{op} \to \mathcal{Set}$ we think of the connections $Q(\gamma_{j,n}): Q_{n-1} \to Q_n$ as extra degeneracies. We shall abuse notation and denote the face maps of $Q$ by $\partial^{\epsilon}_{j}= Q(\delta^{\epsilon}_{j,n}): Q_{n+1} \to Q_n$. We refer to \cite{BrHi81} and \cite{RiZe16} for further details.

Let $\square^{n}_c := \text{Hom}_{\square_c}( \_ , \mathbf{1}^n )$ be the standard $n$-cube with connections.  The category of cubical sets with connections forms a (non-symmetric) monoidal category $(\mathcal{Set}_{\square_c}, \otimes)$ with monodical structure given by
 \begin{equation*}
Q \otimes P:= \underset{\sigma: \square_c^n \to Q, \tau: \square_c^m \to P} {\text{colim}} \square^{n+m}_c.
\end{equation*}
Let $\mathcal{Cat}_{\square_c}$ denote the category of categories enriched over $(\mathcal{Set}_{\square_c}, \otimes)$. Define 
$$\mathfrak{C}_{\square_c}: \mathcal{Set}_{\Delta} \to \mathcal{Cat}_{\square_c}$$
by letting the objects of $\mathfrak{C}_{\square_c}(S)$ to be the elements of $S_0$ and for any two $x,y \in S_0$ the cubical set with connections $\mathfrak{C}_{\square_c}(S)(x,y)$ is given by replacing $(\Delta^1)^{V_T}$ with $\square_c^{V_T}$ in formula \ref{necs}, namely
\begin{equation} \label{cubes}
\mathfrak{C}_{\square_c}(S)(x,y) := \underset{T \to S \in (Nec \downarrow S)_{x,y} }{\text{colim}}  \square_c^{V_T}.
\end{equation}
This colimit is well defined since $T \mapsto \square_c^{V_T}$ defines a functor $Nec \to \mathcal{Set}_{\square_c}$, as explained in section 4 of \cite{RiZe16}. 

Composition in $\mathfrak{C}_{\square_c}$ is defined as follows. For any $x, y \in S_0$, each cube $\eta \in \mathfrak{C}_{\square_c}(S)(x,y)_n$ may be represented as an equivalence class $\eta=[g: T \to S, \zeta]$, where $(g: T\to S) \in (Nec \downarrow S)_{x,y}$ and $\zeta \in (\square_c^{V_T})_n$. Define a composition rule
$$(\mathfrak{C}_{\square_c}(S)(y,z) \otimes \mathfrak{C}_{\square_c}(S)(x,y))_n \to \mathfrak{C}_{\square_c}(S)(x,y)_n$$
to be the map induced by 
$$\mathfrak{C}_{\square_c}(S)(y,z)_k \otimes \mathfrak{C}_{\square_c}(S)(x,y)_l \to \mathfrak{C}_{\square_c}(S)(x,y)_{k+l}$$
given by
$$[g': T' \to S, \zeta'] \otimes  [g: T \to S, \zeta] \mapsto [g \vee g': T \vee T' \to S, \zeta \otimes \zeta' ]$$

\begin{remark} Any cube $\eta \in (\mathfrak{C}_{\square_c}(S)(x,x))_n$ may be represented as the equivalence class $\eta= [g_{\eta}: T_{\eta} \to S, \iota_n]$ where $(g_{\eta}: T_{\eta} \to S ) \in (Nec \downarrow S)$, $V_{T_{\eta}}=n$, and $\iota_n$ is the top dimensional  non-degenerate simplex in $(\square_c^n)_n$.
\end{remark}

Consider the functor $Tr: \mathcal{Set}_{\square_c} \to \mathcal{Set}_{\Delta}$, called triangulation, determined by $Tr(\square^n_c)= (\Delta^1)^{\times n}$, so that for any $K \in \mathcal{Set}_{\square_c}$, $$Tr(K) := \underset{\square_c^n \to K}{\text{colim}} (\Delta^1)^{\times n}$$. Since a natural map (in fact, an isomorphism) $Tr(K) \times Tr(L) \to Tr(K \otimes L)$ for any $K,L \in \mathcal{Set}_{\square_c}$ we obtain an induced functor $$\mathfrak{Tr}: \mathcal{Cat}_{\square_c} \to \mathcal{Cat}_{\Delta}$$ by applying $Tr$ at the level of morphisms. We record a statement which follows immediately from formulas \ref{necs} and \ref{cubes}.
\begin{proposition} There is a natural isomorphism of functors $\mathfrak{C} \cong \mathfrak{Tr} \circ \mathfrak{C}_{\square_c}$.
\end{proposition}

Let $C_{\square}: \mathcal{Set}_{\square_c} \to \mathcal{Ch}_{\mathbf{k}}$ denote the normalized cubical chain complex. More precisely, for any $Q \in Set_{\square_c}$,  consider $C'_{\square}(Q)$ the chain complex with $(C'_{\square}(Q))_n$ defined to be the free $\mathbf{k}$-module generated by elements of $Q_n$ with differential given on any $\sigma \in Q_n$ by $\partial (\sigma):= \sum_{j=1}^n(-1)^{j}(\partial^1_{j}(\sigma) - \partial^0_{j}(\sigma))$. Let $D_nQ$ be the submodule of $(C'_{\square}(Q))_n$ which is generated by those cells in $Q_n$ which are the image of a degeneracy or of a connection map $Q_{n-1} \to Q_n$. The graded module $D_*Q$ forms a subcomplex of $C'_{\square}(Q)$. Then $C_{\square}(Q)$ is defined to be the quotient complex $C'_{\square}(Q) /D_*(Q)$.

\begin{remark} In \cite{BGSW18} it is shown that the chain complex of a cubical set with connections obtained by modding out by degenerate cubes is quasi-isomorphic to the chain complex obtained by modding degenerate cubes together with the images of connections; so these two normalization procedures yield chain complexes which compute the same homology. Note that, in general, the unnormalized chain complex of a cubical set $K$ does not calculate the singular homology of  the geometric realization $|K|$, so a normalization step is always needed. In particular, the un-normalized cubical chain complex of a point will have non-zero higher homology groups. In the appendix of \cite{BGSW18} it is shown that the homology of the normalized chain complex of a cubical set with connections $K$ (defined in either way by modding out connections or not) computes the singular homology of $|K|$, the geometric realization of $K$, which may also be defined, up to homotopy equivalence, in two ways: using the connections as glueing data or not using them. 
\end{remark}

For any $S \in \mathcal{Set}^0_{\Delta}$, $\mathfrak{C}_{\square_c}(S) \in \mathcal{Cat}_{\square_c}$ has a single object $x \in S_0$. Since  $C_{\square}: \mathcal{Set}_{\square_c} \to \mathcal{Ch}_{\mathbf{k}}$ is a monoidal functor there is a dg algebra structure on  $C_{\square}(\mathfrak{C}_{\square_c}(S)(x,x))$  with product induced by the composition rule in $\mathfrak{C}_{\square_c}(S)$. Denote this dg algebra by $\Lambda(S)$. The following proposition may be found as Theorem 7.1 in \cite{RiZe16}. We include a sketch of its proof for completeness. 

\begin{proposition} 
For any $S \in \mathcal{Set}^0_{\Delta}$ there is an isomorphism of dg algebras $\Lambda(S) \cong \Omega (C(S)).$
\end{proposition}

\textit{Sketch of Proof.} The isomorphism $\varphi: \Lambda(S) \to \Omega (C(S))$  is induced by the rule
$$\varphi: (\sigma: \Delta^n \to S) \mapsto [\sigma] \in \Omega (C'(S)) \text{ if  }n>1$$ and 
$$\varphi: (\sigma: \Delta^n\to S)  \mapsto [\sigma] + 1_{\mathbf{k}} \in \Omega (C'(S)) \text{  if  }n=1$$
and extending the resulting $\mathbf{k}$-linear map as an algebra map to necklaces of arbitrary length. A priori, this construction defines a dg algebra map $$\varphi: C'_{\square}(\mathfrak{C}_{\square_c}(S)(x,x)) \to \Omega (C'(S)),$$  where $C'(S)$ denotes the dg coalgebra of un-normalized simplicial chains on $S$. It induces a well defined map after normalizing on both sides since if an $n$-cube $\eta \in (\mathfrak{C}_{\square_c}(S)(x,x))_n$ is the image of a degeneracy or a connection then in the  representation $\eta=[g_{\eta}: T_{\eta} \to S, \iota_n]$ described in the above remark, the restriction of $g_{\eta}$ to some bead of $T_{\eta}$ will define a degenerate simplex in $S$ (this fact follows from Proposition 4.2 in \cite{RiZe16}). The resulting induced map $\varphi: \Lambda(S) \to \Omega (C(S))$ is clearly an isomorphism of algebras. \hfill \qed
\\
\\
We deduce Theorem 1 from the above propositions.
\\
\\
\textit{Proof of Theorem 1.} We argue that the functor $C^K: \mathcal{Set}^0_{\Delta} \to \mathcal{C}_{\mathbf{k}}$, $C^K(S)= C(K(S))$ sends weak homotopy equivalences to $\Omega$-quasi-isomorphisms. Let $f: S \to S'$ be a weak homotopy equivalence where $S, S' \in \mathcal{Set}^0_{\Delta}$ with $\{x \} =S_0$ and $\{x'\} = S'_0$. We have an induced weak homotopy equivalence of Kan complexes, which by abuse of notation we denote by $f: K(S) \to K(S')$. By Proposition 5, $\mathfrak{C}(f): \mathfrak{C}(K(S)) \to \mathfrak{C}(K(S'))$ is a weak equivalence of simplicial categories, so $\mathfrak{C}(f): \mathfrak{C}(K(S))(x,x) \to \mathfrak{C}(K(S'))(x',x')$ is a weak homotopy equivalence of simplicial monoids. Applying the normalized simplicial chains (monoidal) functor $\bar{C}: \mathcal{Set}_{\Delta} \to \mathcal{Ch}_{\mathbf{k}}$, we obtain a quasi-isomorphism of dg algebras $\Gamma(f): \Gamma(K(S)) \to \Gamma(K(S'))$. 

We now describe an explicit quasi-isomorphism of dg algebras $\Phi: \Lambda(S) \to \Gamma(S)$ for any simplicial set $S \in \mathcal{Set}^0_{\Delta}$ with $S_0=\{x\}$. Any generator $\eta$ of the un-normalized cubical chain complex $C_{\square}'(\mathfrak{C}_{\square_c}(S)(x,x))_n$ may be represented as an equivalence class $\eta= [g_{\eta}: T_{\eta} \to S, \iota_n]$ where $(g_{\eta}: T_{\eta} \to S ) \in (Nec \downarrow S)$, $V_{T_{\eta}}=n$, and $\iota_n$ is the top dimensional non-degenerate generator in $C_{\square}(\square_c^n)_n$. Define $$\Phi(\eta) := [g_{\eta}: T_{\eta} \to S, e^{\times n}] \in C'_n(\mathfrak{C}(S)(x,x)),$$ where $C'_n(\mathfrak{C}(S)(x,x))$ denotes the un-normalized simplicial $n$-chains on the simplicial set $\mathfrak{C}(S)(x,x)$ and $e^{\times n} \in \bar{C}_n((\Delta^1)^{\times n})$ denotes the $n$-th fold Eilenberg-Zilber product $e^n = e \times ... \times e$ of the top dimensional generator $e \in \bar{C}_1(\Delta^1)$. In other words, the chain $e^{\times n}$ is the (signed) sum of the $n!$ non-degenerate $n$-simplices of $(\Delta^1)^{\times n}$, so that $[g_{\eta}: T \to S, e^{\times n}] \in \Gamma(S)$ is also a sum of $n!$ generators in $\Gamma(S)$. If $\eta$ happens to be a degenerate cube or the image of a connection then the restriction of $g_{\eta}: T_{\eta} \to S$ to at least one bead of $T_{\eta}$ determines a degenerate simplex in $S$. This implies that $\Phi$ induces a map $\Lambda(S) \to \Gamma(S)$ after passing to normalizations. This map preserves the algebra structures since for any two cubes $\eta$ and $\eta'$ of dimension $k$ and $l$ respectively, 
\begin{eqnarray*} \Phi( [g_{\eta} \vee g_{\eta'} : T_{\eta} \vee T_{\eta'} \to S,  \iota_k \otimes \iota_l] )
\\ = [g_{\eta} \vee g_{\eta'} : T_{\eta} \vee T_{\eta'} \to S, e^{\times( k+l)}= e^{\times (k+l)} ]
\\
 = [g_{\eta'}: T_{\eta'} \to S, e^{\times l}] \circ [g_{\eta}: T_{\eta} \to S, e^{\times k}].
 \end{eqnarray*}
An acyclic models argument, using $\square_c^n$ and $(\Delta^1)^{\times n}$ as models, implies that $\Phi$ is a quasi-isomorphism (see Lemma 7.2 of \cite{RiZe16}).

By Proposition 9, $\Lambda(K(S)) \cong \Omega (C(K(S)))=\Omega (C^K(S))$, so we have a natural quasi-isomorphism of dg algebras $\Gamma(K(S))\simeq \Omega (C^K(S))$ and similarly $\Gamma(K(S')) \simeq \Omega (C^K(S'))$. By the naturality of the constructions we conclude that $\Omega(f): \Omega (C^K(S)) \to \Omega (C^K(S'))$ is a quasi-isomorphism of dg algebras, as desired.

Now we deduce part b) of Theorem 1. By Corollary 7 it follows that we have a weak homotopy equivalence of simplcial monoids $\mathfrak{C}(K(S))(x,x) \simeq \textbf{Sing}(\mathbf{\Omega} |S|) (x,x)$ which induces a quasi-isomorphism of dg algebras 
\begin{equation}
\Gamma(K(S)) \simeq \bar{C}( \textbf{Sing}(\mathbf{\Omega} |S| ) (x,x))= \bar{C}( \text{Sing}(\Omega_x |S|)).
\end{equation} By the argument in the above paragraph we have $\Gamma(K(S)) \simeq \Omega (C^K(S))$. Hence, there exists a quasi-isomorphism of dg algebras 
$$\Omega (C^K(S)) \simeq \bar{C}(\text{Sing}( \Omega_x|S|)),$$ as desired. \hfill $\qed$

\subsection{Theorem 2: determining the fundamental group from a dg bialgebra structure on the cobar construction}

For any path connected pointed space $(X,b)$ we have a natural isomorphism of algebras $H_0(\Omega_bX; \mathbf{k}) \cong \mathbf{k}[\pi_1(X,b)]$ between the $0$-th homology group of the based loop space and the fundamental group algebra. Hence, Theoremx 1 implies the following.
\begin{corollary} For any $S \in \mathcal{Set}^0_{\Delta}$ there is a natural isomorphism of $\mathbf{k}$-algebras $H_0(\Omega (C^K(S))) \cong \mathbf{k}[\pi_1(S)]$. 
\end{corollary}
The isomorphism in Corollary 10 is given explicitly by sending a generator $[\sigma]$ representing a class in $H_0(\Omega (C^K(S)))$ to the class in $\mathbf{k}[\pi_1(S)]$ represented by $\sigma -1_{\mathbf{k}}$ and then extending as an algebra map on monomials of arbitrary length.  

The group ring $\mathbf{k}[\pi_1(S)]$ has a natural Hopf algebra structure: coproduct is determined by $\gamma \mapsto \gamma \otimes \gamma$,  counit by $\epsilon(\gamma)=1_{\mathbf{k}}$, and antipode by $s(\gamma) = \gamma^{-1}$ for all $\gamma \in \pi_1(S)$. Moreover, there is a dg bialgebra structure on $\Omega(C^K(S))$, originally described in \cite{Ba98} in the simply connected case, inducing a Hopf algebra structure on $H_0(\Omega(C^K(S))$ which agrees with that of $\mathbf{k}[\pi_1(X,b)]$ along the above isomorphism. The coassociative coproduct on $\Omega(C^K(S))$ is obtained from the cubical structure as we recall know. 

For any $Q \in \mathcal{Set}_{\square_c}$, we may construct a natural coassociative coproduct $\Delta_{\square}: C_{\square}(Q) \to C_{\square}(Q) \otimes C_{\square}(Q)$ induced by the formula 
\begin{equation*}
\Delta_{\square}(\sigma) = \sum (-1)^{\epsilon} \partial_{j_1}^0 ... \partial^0_{j_p} (\sigma) \otimes \partial^1_{i_1} ... \partial^1_{i_q}(\sigma),
\end{equation*}
where $\sigma \in Q_n$, the sum runs through all shuffles $\{ i_1 < ... < i_q, j_1 < ... < j_p \}$ of $\{1, ..., n\}$ and $(-1)^{\epsilon}$ is the shuffle sign. If $\sigma \in Q_0$ set $\Delta_{\square}(\sigma)= \sigma \otimes \sigma$. The above formula is determined by declaring $\Delta_{\square}(\sigma)= \partial^0_0 (\sigma) \otimes \sigma + \sigma \otimes \partial^1_0(\sigma)$ for any $1$-cube $\sigma \in Q_1$ (same formula as the Alexander-Whitney diagonal of a $1$-simplex) and then extending it using the fact that an $n$-cube is a product of $n$ $1$-cubes. 

If $Q$ is a cubical monoid with connections then $C_{\square}(Q)$ becomes dg bialgebra, i.e. $\Delta_{\square}$ is multiplicative. Moreover, the dg bialgebra structure is unital and counital: the unit is determined by sending $1_\mathbf{k}$ to the identity of the monoid, and the counit $\epsilon: C_{\square}(Q) \to \mathbf{k}$ by sending all generators in $C_{\square}(Q)_0$ to $1_{\mathbf{k}}$ and everything else to $0$. Note $\epsilon$ is also an augmentation for the algebra structure. Applying this construction to the cubical monoid with connections $\mathfrak{C}_{\square}(S)(x,x)$ associated to any $S \in \mathcal{Set}^0_{\Delta}$ we obtain the following.

\begin{proposition} For any $S \in \mathcal{Set}^0_{\Delta}$ the dg algebra structure on $\Omega(C(S))$ extends to a (unital and counital) dg bialgebra structure. 
\end{proposition}
\begin{proof}
By the discussion in the previous paragraphs we have that the dg algebra structure on $\Lambda(S):= C_{\square}(\mathfrak{C}_{\square_c}(S)(x,x))$ extends to a dg bialgebra structure when equipped with the coproduct $\Delta_{\square}$. The result now follows from Proposition 9.
\end{proof}

Denote the coproduct on $\Omega(C(S))$ by $\nabla: \Omega(C(S)) \to \Omega(C(S)) \otimes  \Omega(C(S))$, the unit by $\eta: \mathbf{k} \to \Omega(C(S))$, and the counit by $\epsilon: \Omega(C(S)) \to \mathbf{k}$. We also denote the multiplication in $\Omega(C(S))$ by $\mu$. Note that $H_0(\Omega(C^K(S)))$ inherits a bialgebra structure and we denote the induced structure maps by the same symbols $\nabla, \mu, \eta, \epsilon$. 

\begin{proposition} For any $S \in \mathcal{Set}^0_{\Delta}$ the bialgebra $H_0(\Omega(C^K(S)))$ is a Hopf algebra, i.e. there exists a natural linear map $s: H_0(\Omega(C^K(S)))\to H_0(\Omega(C^K(S)))$ satisfying $\mu \circ (s \otimes id) \circ \nabla = \eta \circ \epsilon = \mu \circ (id \otimes s) \circ \nabla$. Moreover, there is an isomorphism of Hopf algebras $H_0(\Omega(C^K(S)))\cong \mathbf{k}[\pi_1(S)]$.
\end{proposition}

\begin{proof} The antipode is induced by the map that sends a $0$-cycle $[\sigma] \in \Omega(C^K(S))$ to $[\sigma^{-1}]$ where $\sigma^{-1} \in \text{Sing}({|S|,x})_1$ is defined by running the loop $\sigma: \Delta^1 \to |S|$ with the opposite orientation. It is straightforward to check this induces a well defined map $s: H_0(\Omega(C^K(S)))\to H_0(\Omega(C^K(S)))$, it satisfies the antipode equation, and that the isomorphism $H_0(\Omega(C^K(S)))\cong \mathbf{k}[\pi_1(S)]$ of Corollary 10 preserves the Hopf algebra structures.
\end{proof}

We deduce Theorem 2 from Theorem 1 and the above propositions.
\\
\\
\textit{Proof of Theorem 2.} For any $S \in \mathcal{Set}_{\Delta}^0$, define $B(S)$ to be the dg coassociative coalgebra $(C^K(S), \partial, \Delta)$ together with the dg bialgebra structure on $\Omega(C^K(S))$ provided by Proposition 11. Proposition 12 tells us that such a dg bialgebra structure has the property that its $0$-th homology admits an antipode making it a Hopf algebra. By the naturality of our constructions it follows that $B(S)$ is functorial on $S$ so it defines a functor $B: \mathcal{Set}^0_{\Delta} \to \mathcal{B}_{\mathbf{k}}$. The fact that $B$ sends weak homotopy equivalences to $\Omega$-quasi-isomorphisms follows from Theorem 1. 

 The functor  $F: (\mathcal{B}_{\mathbf{k}}, \Omega\text{-}q.i.) \to  (\mathcal{C}_{\mathbf{k}}, \Omega\text{-}q.i.)$ is defined by sending any $(C, \partial, \Delta, \nabla, \eta, \epsilon)$ to its underlying connected dg coassociative coalgebra $(C, \partial, \Delta)$. 

The existence of a natural isomorphism of functors $\pi_1 \cong Grp \circ H_{\Omega} \circ B$ follows from the second part of Proposition 12 and the fact that for any group $G$ we have a natural isomorphism $Grp(\mathbf{k}[G]) \cong G$. 
\hfill $\qed$

\subsection{Theorem 3: adding the $E_{\infty}$-structure} We deduce Theorem 3 from constructions in \cite{BeFr04} and observations of \cite{Ka03}. We would like to stress the importance of Kan replacing and using the notion of $\Omega$-quasi-isomorphisms of $\chi$-coalgebras in order to keep the data of the fundamental group: as a consequence of Theorem 3 one may recover both the $E_{\infty}$-structure and the fundamental group of $S$ from any $\chi$-coalgebra $\Omega$-quasi-isomorphic to $C(K(S))$ with the $\chi$-coalgebra structure described in  \cite{BeFr04}. In general, this statement is false if we do not preform a Kan replacement or some form of localization, or if we use the weaker notion of quasi-isomorphisms of $\chi$-coalgebras instead of $\Omega$-quasi-isomorphisms. 

Recall the surjection operad $\chi$ is a dg operad defined as a sequence of $\mathbf{k}$-chain complexes $\{ \chi(1), \chi(2), .... \}$ such that $\chi(r)_d$ is generated as a $\mathbf{k}$-module by surjections of sets $u: \{1,...,r+d\} \to \{1,...,r\}$ which are non-degenerate, i.e. satisfying $u(i) \neq u(i+1)$. A non-degenerate surjection $u: \{1,...,r+d\} \to \{1,...,r\}$ is usually denoted by the string $(u(1),...,u(r+d))$.    For the definition of the differential $\chi(r)_d \to \chi(r)_{d-1}$ and further details we refer to \cite{McSm02} and \cite{BeFr04}. The surjection operad is an $E_{\infty}$-operad, namely,  $\chi$ is quasi-isomorphic to the commutative operad and the action of the symmetric group $\Sigma_r$ on each $\chi(r)$ yields a projective $\Sigma_r$-module.

The surjection operad can be given a filtration $F_1\chi \subset F_2\chi \subset ... \subset \chi$ where $F_n\chi$ is quasi-isomorphic to the normalized singular chains on the (topological) little $n$-cubes operad; in particular, $F_1\chi$ is the associative operad. $F_2 \chi$ is generated by the surjections $(1,2)$, $(1,2,1)$, $(1,2,1,3,1)$, $(1,2,3,...,1,q,1),...$
\\

\textit{Proof of Theorem 3.}
Associated to any $S \in \mathcal{Set}^0_{\Delta}$ there is a natural $\chi$-coalgebra structure on $C^K(S)=C(K(S))$ extending the connected dg coalgebra given by the Alexander-Whitney coassociative coproduct, as described in \cite{BeFr04}. We denote this $\chi$-coalgebra by $E(S)$. Moreover, among the structure maps of $E(S)$, those corresponding to $F_2{\chi}$ define a homotopy $G$-coalgebra on $C^K(S)$, as identified explicitly in \cite{Ka03}. As explained by Kadeishvili, a homotopy $G$-coalgebra structure extending a dg coalgebra gives rise to a dg bialgebra on the cobar construction. More precisely, given any connected $\chi$-coalgebra $\mathbf{C}$, each surjection $(1,2,3,...,1,q,1) \in F_2\chi$ corresponds to a map 
$$E^{1,q}: \mathbf{C} \to \mathbf{C} \otimes \mathbf{C}^{\otimes q}.$$
The direct sum of all $E^{1,q}$ induces a coproduct on $\nabla_{\mathbf{C}}: \Omega C \to \Omega C \otimes \Omega C$, making $\Omega C$, the cobar construction on the underlying connected dg coassociative coalgebra $(C, \partial, \Delta)$ of $\mathbf{C}$, into a unital counital dg bialgebra with counit $\epsilon: \Omega C \to \mathbf{k}$ the canonical projection. The functor $G: (\mathcal{E}_{\mathbf k},\Omega\text{-}q.i.) \to (\mathcal{B}_{\mathbf{k}}, \Omega\text{-}q.i.)$ is defined by sending $\mathbf{C}$ to $(C,\partial, \Delta, \nabla_{\mathbf{C}}, \eta, \epsilon)$ where $\eta$ is the unit of the dg algebra $\Omega C$. In the case of $C^K(S)$ the dg bialgebra structure on $\Omega(C^K(S))$ is naturally isomorphic to $B(S)$, namely, the dg counital coalgebra $(\Omega(C^K(S)), \nabla, \epsilon)$ obtained in Proposition 11  is naturally isomorphic to $(\Omega(C^K(S)), \nabla_{E(S)}, \epsilon)$. Thus, this construction defines a functor $E: \mathcal{Set}^0_{\Delta} \to \mathcal{E}_{\mathbf{k}}$ satisfying $G \circ E  \cong B$. 

Furthermore, $E$ sends weak homotopy equivalences to $\Omega$-quasi-isomorphisms. This follows since the underlying dg coalgebra structure of $E(S)$ is precisely $C^K(S)$ and, by Theorem 1, $C^K$ sends weak homotopy equivalences to $\Omega$-quasi-isomorphisms. 
 \hfill $\qed$
\begin{remark} For a detailed treatment of the signs involved the bialgebra structure on the cobar construction of a $F_2\chi$-coalgebra constructed in \cite{Ka03}, see \cite{Yo13}. For a proof relating Baues coproduct on the cobar construction and Kadeishvili's construction in the simply connected case, see \cite{Qu16}. 
\end{remark}

\subsection*{Acknowledgements} MR acknowledges the support of the grant Fordecyt 265667 and the excellent working conditions in CINVESTAV and \textit{Centro de colaboraci\'on Samuel Gitler} in Mexico City where parts of this research were conducted. MZ would like to thank Owen Gwilliam and Gregory Ginot for helpful conversations and the \textit{Max Planck Institute for Mathematics} for their support and hospitality during his visits. Both authors would like to thank Michael Mandell,  Anibal Medina for suggesting a diagram of functors that would be helpful in summarizing the results, and Dennis Sullivan for insightful discussions.

\bibliographystyle{plain}

 \Addresses

\end{document}